\documentclass[a4paper,reqno]{amsart}

\usepackage{amssymb}
\usepackage{latexsym}
\usepackage{amsmath}
\usepackage{amsthm}
\usepackage{bbm}

\usepackage{pgfplots}
\pgfplotsset{compat=1.15}
\usepackage{mathrsfs}
\usetikzlibrary{arrows}
\definecolor{ududff}{rgb}{0.30196078431372547,0.30196078431372547,1}
\definecolor{zzttqq}{rgb}{0.6,0.2,0}
\definecolor{xdxdff}{rgb}{0.49019607843137253,0.49019607843137253,1}
\definecolor{ududff}{rgb}{0.30196078431372547,0.30196078431372547,1}

\usepackage{hyperref}
\usepackage{xcolor}
\usepackage{amsmath}

\def\dist{{\text{\rm dist}}}

\DeclareMathOperator{\R}{\mathbb{R}}

\def\dd{\mathrm{d}}

\renewcommand\setminus{\mathbin{\mathpalette\rsetminusaux\relax}}
\newcommand\rsetminusaux[2]{\mspace{-4mu}
  \raisebox{\rsmraise{#1}\depth}{\rotatebox[origin=c]{-20}{$#1\smallsetminus$}}
 \mspace{-4mu}
}
\newcommand\rsmraise[1]{%
  \ifx#1\displaystyle .8\else
    \ifx#1\textstyle .8\else
      \ifx#1\scriptstyle .6\else
        .45%
      \fi
    \fi
  \fi}

\makeatletter
\newcommand*{\rom}[1]{\expandafter\@slowromancap\romannumeral #1@}
\makeatother

\DeclareMathOperator{\inrad}{inrad}
\DeclareMathOperator{\diam}{diam}

\newtheorem{theorem}{Theorem}[section]
\newtheorem*{thm*}{Theorem}

\newtheorem{lemma}[theorem]{Lemma}

\theoremstyle{definition}

\theoremstyle{definition}

\newtheorem{remark}[theorem]{Remark}

\numberwithin{equation}{section}

\title[Hot-spots free subregions]{A note on hot-spots free subregions of convex domains}

\author[J.~Rohleder]{Jonathan Rohleder}
\address{Matematiska institutionen \\ Stockholms universitet \\
106 91 Stockholm \\
Sweden}
\email{jonathan.rohleder@math.su.se}

\dedicatory{For Henk de Snoo on the occasion of his 80th birthday}



\begin{document}

\begin{abstract}
The second eigenfunction of the Neumann Laplacian on convex, planar domains is considered. Inspired by the famous hot spots conjecture and a related result of Steinerberger, we show that potential critical points of this eigenfunction (and, in particular, interior ``hot spots'') cannot be located ``near the center'' of the domain. The region in which critical points are excluded is described explicitly.
\end{abstract}

\maketitle

\section{Introduction}

The famous hot spots conjecture of Jeffrey Rauch suspects eigenfunctions corresponding to the second eigenvalue $\mu_2$ of the Laplacian with Neumann boundary conditions on a bounded domain $\Omega \subset \R^d$ to attain their maximum and minimum only on the boundary. This conjecture is by now known to fail for certain multiply connected domains in $\R^2$ \cite{B05,BW99} as well as for convex domains in high dimensions \cite{DP24+}. On the other hand, the conjecture has been proven for certain classes of planar domains such that domains with symmetries \cite{JN00}, longish domains \cite{AB04,BB99,KT19,R21+} and triangles \cite{JM20,JM20err}. It is widely believed that the hot spots conjecture should be true for any convex domain in two dimensions.

Though proving the hot spots conjecture for general convex planar domains is a hard problem, certain complementary results have been obtained in recent years, especially restricting the possible location of maxima and minima for any second eigenfunction of the Neumann Laplacian. A recent result in this direction is the following theorem.
Here, for a convex bounded domain $\Omega \subset \R^2$, $\psi_2 \in H^2 (\Omega)$ is a non-trivial function satisfying
\begin{align*}
\left\{
\begin{array}{r c l l}
 - \Delta \psi_2 & = & \mu_2 \psi_2 & \text{in}~\Omega, \\[1mm]
 \nabla \psi_2 \cdot \nu & = & 0 & \text{on}~\partial \Omega,
\end{array}
\right.
\end{align*}
where $\nu$ is the exterior unit normal vector field on $\partial \Omega$ (defined almost everywhere) and $\mu_2$ is the second, i.e.\ first non-zero, eigenvalue of the Neumann Laplacian.

\begin{thm*}[Steinerberger \cite{S20}]
There exists a universal constant $c > 0$ such that for any convex bounded domain $\Omega \subset \R^2$, if $\mathbf{x, y} \in \partial \Omega$ are at maximal distance, i.e., $|\mathbf{x} - \mathbf{y}| = \diam (\Omega)$, then $\psi_2$ attains any global maximum at a distance at most $c \cdot \inrad (\Omega)$ from $\{ \mathbf{x}, \mathbf{y}\}$, where $\inrad (\Omega)$ denotes the inradius of $\Omega$.
\end{thm*}

In other words, the result confirms that the global maxima (and minima) of $\psi_2$ have to lie within a certain distance to the ``tips'' of the domain -- whether they are on the boundary or not. As Steinerberger points out, his proof, based on Brownian motion techniques, is able to produce an upper bound on $c$, even though in certain simple examples it would give some $c \geq 2$. 

The aim of this note is to provide a result of a somewhat similar flavor. It deals with points in the interior of $\Omega$ only but, on the other hand, is more explicit in its character. The techniques which we will use in its proof are not novel. In fact, they have been around for quite a while, even in the context of the hot spots conjecture. They go back at least to work of Miyamoto \cite{M09}, where he proves that the conjecture is true for domains ``close to a disk'', and were recently used by Hatcher in the context of mixed boundary conditions \cite{H24+}. However, to the best of our knowledge it has remained unnoticed so far that a slight modification leads to the following result. In what follows we denote by $j_0 \approx 2.4048$ and $j_1 \approx 3.8317$ the first positive zero of the Bessel function $J_0$ and $J_1$ of order zero and one, respectively. 

\begin{theorem}\label{thm:main}
Let $\Omega \subset \R^2$ be a convex, bounded open set with diameter $\diam (\Omega)$. If $\psi_2$ has a critical point $\mathbf{x_0}$ in $\Omega$, then there exists $\mathbf{y} \in \partial \Omega$ such that
\begin{align*}
 \dist (\mathbf{x_0, y}) > \frac{j_1}{2 j_0} \diam (\Omega) \approx 0.7967 \diam (\Omega),
\end{align*}
where $\dist$ denotes the Euclidean distance in $\R^2$. 
\end{theorem}

Roughly speaking, Theorem \ref{thm:main} excludes critical points of $\psi_2$ ``near the center'' of $\Omega$. Since interior global extrema of $\psi_2$ are, in particular, critical points, this implies a restriction on the possible hot spots of any $\psi_2$. For instance, if $\Omega$ is not far from a disk, then the theorem prohibits critical points near the center. On the other hand, if $\Omega$ is rather longish, it indeed asserts that all possible critical points must lie close to the tips. In Figure \ref{fig:examples}, Theorem \ref{thm:main} is illustrated at the examples of a disk and an ellipse. (We point out that in these explicit cases the hot spots conjecture is known to hold true.)

\begin{figure}[h]
\includegraphics[scale=0.3]{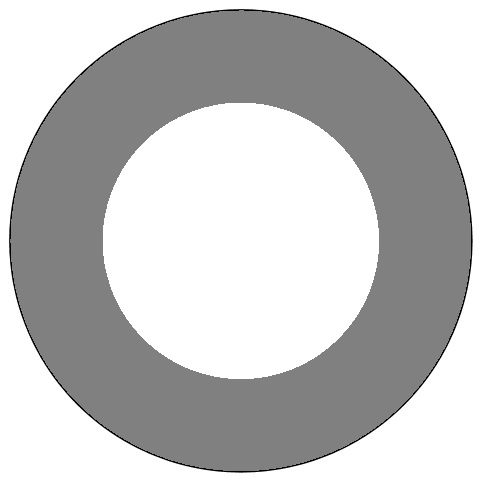} \hspace{10mm}
\includegraphics[scale=0.3]{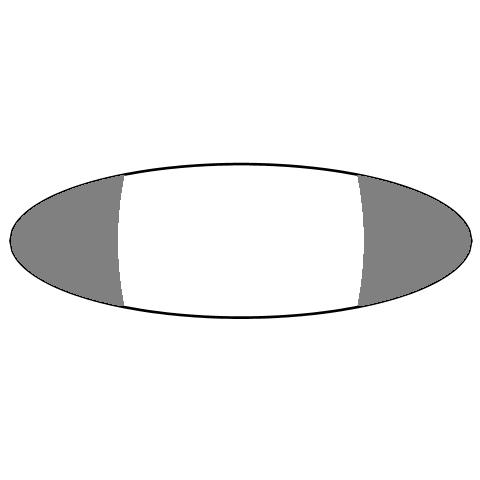}

\caption{A circle and an ellipse; according to Theorem \ref{thm:main}, critical points of $\psi_2$ may only be located in the respective gray region.}
\label{fig:examples}
\end{figure}

The next section contains the proof of Theorem \ref{thm:main}. Miyamoto \cite{M09} uses the nodal structure of solutions to $- \Delta w = \mu_2 w$ near degenerate zeroes in a smart way and combines it with the Szeg\H{o}--Weinberger inequality, i.e.\ with a bound for $\mu_2$ in terms of the area of $\Omega$. In the present note we modify this approach by using a diameter-based bound on $\mu_2$ instead. More precisely, we make use of the inequality
\begin{align}\label{eq:Kroger}
 \mu_2 \leq \frac{4 j_0^2}{\diam (\Omega)^2}
\end{align}
proven by Kr\"oger \cite{K99} for any convex $\Omega \subset \R^2$.

\section{Proof of Theorem \ref{thm:main}}

In this section we prove the main result of this note, based on ideas of Miyamoto \cite{M09}. We will use a few elementary properties of some Bessel functions of the first kind.

Recall that the Bessel function $J_0$ of order zero satisfies the differential equation
\begin{align}\label{eq:Bessel0}
 x^2 \frac{\dd^2 u}{\dd x^2} + x \frac{\dd u}{\dd x} + x^2 u = 0
\end{align}
and the normalization $J_0 (0) = 1$. Furthermore, its derivative is given by $J_0' = - J_1$, the negative of the Bessel function $J_1$ of first order. In particular, $J_0' (0) = 0$ and $J_0' (x) < 0$ on $(0, j_1)$, where $j_1$ denotes the first positive zero of $J_1$.

We will also make use of the following lemma. For completeness we provide its short proof.

\begin{lemma}\label{lem:nodal}
Let $\Omega \subset \R^2$ be a bounded Lipschitz domain and let $\mu_2$ be the second eigenvalue of the Neumann Laplacian on $\Omega$. If $w \in H^1 (\Omega)$ is a non-trivial function satisfying $- \Delta w = \mu_2 w$ in the distributional sense, then $w$ does not have any interior nodal domain; that is, there is no open, non-empty set $\Omega' \subset \Omega$ such that $w |_{\Omega'}$ belongs to $H_0^1 (\Omega')$.
\end{lemma}

\begin{proof}
Assume the converse, i.e., there exists an open, non-empty set $\Omega' \subset \Omega$ such that $u := w |_{\Omega'} \in H_0^1 (\Omega')$. Clearly, $- \Delta u = \mu_2 u$ in $\Omega'$, and, due to unique continuation, $u$ is non-trivial, and, thus, an eigenfunction of the Dirichlet Laplacian on $\Omega'$ with corresponding eigenvalue $\mu_2$. Let us denote by $\lambda_1 (\Omega)$ and $\lambda_1 (\Omega')$ the first Dirichlet Laplacian eigenvalues  on $\Omega$ and $\Omega'$, respectively. Using P\'olya's classical result $\mu_2 < \lambda_1 (\Omega)$ \cite{P52} and domain monotonicity of Dirichlet Laplacian eigenvalues, we obtain
\begin{align*}
 \mu_2 < \lambda_1 (\Omega) \leq \lambda_1 (\Omega') \leq \mu_2,
\end{align*}
a contradiction.
\end{proof}

We can now proceed with the proof of the main observation of this note.

\begin{proof}[Proof of Theorem \ref{thm:main}]
Let us assume that $\mathbf{x_0} \in \Omega$ is a critical point of an eigenfunction $\psi_2$ of the Neumann Laplacian on $\Omega$ corresponding to the eigenvalue $\mu_2$. Because of the translation invariance of the Laplacian we may assume, without loss of generality, that $\mathbf{x_0} = \mathbf{0}$. 

For a contradiction, assume now that 
\begin{align}\label{eq:contradiction}
 |\mathbf{y}| = \dist (\mathbf{0}, \mathbf{y}) \leq \frac{j_1}{2 j_0} \diam (\Omega) \quad \text{for all}~\mathbf{y} \in \partial \Omega.
\end{align}
Possibly after replacing $\psi_2$ by $- \psi_2$, we have
\begin{align}\label{eq:sign}
 \psi_2 (\mathbf{0}) \geq 0.
\end{align}
Using polar coordinates $\mathbf{x} = (r \cos \theta, r \sin \theta)$, let us set 
\begin{align*}
 w (\mathbf{x}) := \psi_2 (\mathbf{0}) J_0 \left( \sqrt{\mu_2} r \right) - \psi_2 (\mathbf{x}), \quad \mathbf{x} \in \Omega.
\end{align*}
From \eqref{eq:Bessel0} we get
\begin{align}\label{eq:Laplace}
\begin{split}
 - \Delta w (\mathbf{x}) & = - \mu_2 \psi_2 (\mathbf{0}) \left( J_0'' (\sqrt{\mu_2} r) + \frac{1}{\sqrt{\mu_2} r} J_0' (\sqrt{\mu_2} r) \right) - \mu_2 \psi_2 (\mathbf{x}) \\
 & = \mu_2 \psi_2 (\mathbf{0}) J_0 (\sqrt{\mu_2} r) - \mu_2 \psi_2 (\mathbf{x}) \\
 & = \mu_2 w (\mathbf{x})
\end{split}
\end{align}
for $\mathbf{x} \in \Omega$. Moreover,
\begin{align*}
 w (\mathbf{0}) = \psi_2 (\mathbf{0}) J_0 (0) - \psi_2 (\mathbf{0}) = 0
\end{align*}
and
\begin{align*}
 \nabla w (\mathbf{0}) = \psi_2 (\mathbf{0}) \sqrt{\mu_2} J_0' (0) \frac{1}{r} \mathbf{x} |_{\mathbf{x} = \mathbf{0}} + \nabla \psi_2 (\mathbf{0}) = 0,
\end{align*}
as $\mathbf{x_0} = \mathbf{0}$ is a critical point of $\psi_2$. It is a consequence of the latter two identities together with \eqref{eq:Laplace} that $w$ is a solution to $- \Delta w = \mu_2 w$ that has a zero of order at least two at $\mathbf{x_0} = \mathbf{0}$. Hence, the nodal set $\{ \mathbf{x} \in \Omega : w (\mathbf{x}) = 0 \}$ of $w$ consists, in a neighborhood of $\mathbf{x_0} = \mathbf{0}$, of at least four branches meeting at $\mathbf{0}$; see, e.g., \cite[Proposition 2.1]{M09} for details and references. Note further that by Lemma \ref{lem:nodal} $w$ has no interior nodal domain. Combining these observations implies that $w$ has at least four distinct nodal domains whose boundary contains $\mathbf{x_0} = \mathbf{0}$ and that $w$ is positive on at least two such nodal domains, which we call $\Omega_1^+$ and $\Omega_2^+$; see Figure \ref{fig:nodal}.
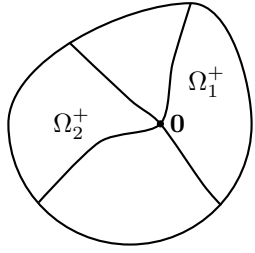
\begin{figure}[h!]
\begin{tikzpicture}[scale=0.8]
 \draw[looseness=0.75,thick] (-3,0) arc (180:360:2) to[out=90,in=0] (0,2) to[out=180,in=90] (-3,0) -- cycle;
 \draw[black,fill=black] (-0.5,0) circle(0.05) node[right,black]{$\mathbf{0}$};
 \draw[thick] plot[smooth] coordinates {(-2.5,-1.3) (-1.5,-0.3) (-0.5,0) (-0.3,1) (0,2) };
 \draw[thick] plot[smooth] coordinates {(-2,1.35) (-1,0.4) (-0.5,0) (0.2,-1) (0.5,-1.35) };
 \draw[white] (-0.2,0.7) circle(0.05) node[right,black]{$\Omega_1^+$};
 \draw[white] (-1.5,0) circle(0.05) node[left,black]{$\Omega_2^+$};
\end{tikzpicture}
\caption{Typical structure of the nodal lines of $w$ in $\Omega$; $w$ is positive on both $\Omega_1^+$ and $\Omega_2^+$.}
\label{fig:nodal}
\end{figure}

Note that for almost all $\mathbf{x} = (r \cos \theta, r \sin \theta) \in \partial \Omega$, the normal derivative $\partial_\nu w = \nabla w \cdot \nu$ satisfies
\begin{align}\label{eq:normalDerivative}
 (\partial_\nu w) (\mathbf{x}) = \psi_2 (\mathbf{0}) \sqrt{\mu_2} J_0' (\sqrt{\mu_2} r) \frac{1}{r} \mathbf{x} \cdot \nu.
\end{align}
Furthermore, using the spectral bound \eqref{eq:Kroger} and the assumption \eqref{eq:contradiction} we get for $\mathbf{x} = (r \cos \theta, r \sin \theta) \in \partial \Omega$
\begin{align*}
 \sqrt{\mu_2} r \leq \sqrt{\mu_2} \sup_{\mathbf{y} \in \partial \Omega} |\mathbf{y}| \leq \frac{2 j_0}{\diam (\Omega)} \frac{j_1}{2 j_0} \diam (\Omega) = j_1,
\end{align*}
and since $J_0' (x) < 0$ for all $x \in (0, j_1)$ it follows $J_0' (\sqrt{\mu_2} r) \leq 0$ if $\mathbf{x} \in \partial \Omega$. Moreover, the convexity of $\Omega$ implies $\mathbf{x} \cdot \nu > 0$ for all $\mathbf{x} \in \partial \Omega$. Especially, \eqref{eq:normalDerivative} gives
\begin{align*}
 (\partial_\nu w) (\mathbf{x}) \leq 0 \quad \text{for all}~\mathbf{x} \in \partial \Omega.
\end{align*}
Now, as $\varphi_j := w |_{\Omega_j^+}$ is non-negative and vanishes on $\partial \Omega_j^+ \setminus \partial \Omega$, $j = 1, 2$, we have
\begin{align}\label{eq:yeah}
 \int_{\Omega_j^+} |\nabla \varphi_j|^2 = \mu_2 \int_{\Omega_j^+} |\varphi_j|^2 + \int_{\partial \Omega_j^+} w \partial_\nu w \leq \mu_2 \int_{\Omega_j^+} |\varphi_j|^2, \quad j = 1, 2.
\end{align}
Let $\widetilde \varphi_j \in H^1 (\Omega)$ denote the extension by zero of $\varphi_j$ to all of $\Omega$, $j = 1, 2$, and let $u := \alpha_1 \widetilde \varphi_1 + \alpha_2 \widetilde \varphi_2$, where $\alpha_1, \alpha_2 \in \R$ are chosen such that $\int_\Omega u = 0$. As $u$ vanishes identically on $\Omega \setminus (\Omega_1^+ \cup \Omega_2^+)$, it is not an eigenfunction of the Neumann Laplacian on $\Omega$ and, hence,
\begin{align}\label{eq:amazing}
\begin{split}
 \mu_2 \int_\Omega |u|^2 & < \int_\Omega |\nabla u|^2 = \alpha_1^2 \int_{\Omega_1^+} |\nabla \varphi_1|^2 + \alpha_2^2 \int_{\Omega_2^+} |\nabla \varphi_2|^2 \\
 & \leq \mu_2 \left( \alpha_1^2 \int_{\Omega_1^+} |\varphi_1|^2 + \alpha_2^2 \int_{\Omega_2^+} |\varphi_2|^2 \right),
\end{split}
\end{align}
by \eqref{eq:yeah} and the fact that $\varphi_1, \varphi_2$ have disjoint supports. On the other hand, the left-hand side of \eqref{eq:amazing} equals
\begin{align*}
 \mu_2 \int_\Omega |u|^2 = \mu_2 \left( \alpha_1^2 \int_{\Omega_1^+} |\varphi_1|^2 + \alpha_2^2 \int_{\Omega_2^+} |\varphi_2|^2 \right),
\end{align*}
leading \eqref{eq:amazing} to a contradiction. This completes the proof.
\end{proof}

\begin{remark}
The essentially same proof also leads to the following result: if $\Omega \subset \R^2$ is a convex, bounded open set such that 
\begin{align}\label{eq:strongKroger}
 \mu_2 \leq \frac{j_1^2}{\diam (\Omega)^2}
\end{align}
holds, then $\psi_2$ has no critical points in $\Omega$ and, in particular, the hot spots conjecture is true for $\Omega$. The condition \eqref{eq:strongKroger} is satisfied for certain convex domains. For instance, the Payne--Weinberger inequality \cite{PW60} asserts that the sharp inequality
\begin{align*}
 \mu_2 \geq \frac{\pi^2}{\diam (\Omega)}
\end{align*}
holds for any convex domain, with equality in the limit when $\Omega$ degenerates to a one-dimensional interval. Since $\frac{j_1}{\pi} \approx 1.2197$, \eqref{eq:strongKroger} is satisfied when $\Omega$ is a convex domain sufficiently close to an interval.
\end{remark}

%

\subsection*{Competing interests}

No competing interests prevail.


\subsection*{Acknowledgments}

The author acknowledges financial support by the grant no.\ 2022-03342 of the Swedish Research Council (VR).

\end{document}